\documentclass[a4paper, twoside]{article}
\usepackage[utf8]{inputenc}
\usepackage[english]{babel}
\usepackage[T1]{fontenc}
\usepackage{a4wide}

\usepackage{amsmath, latexsym, amsfonts, amssymb, amsthm, amscd}
\usepackage{yfonts}
\usepackage{mathdots}
\usepackage{dsfont}
\usepackage{titlesec}

\newcounter{c}

\titleformat{\section}{\normalfont\scshape\center}{\thesection}{1em}{}
\titleformat{\subsection}{\normalfont\scshape}{\thesubsection}{1em}{}
\titleformat{\subsubsection}{\normalfont\scshape}{\thesubsubsection}{1em}{}

\newtheorem{lem}{Lemma}
\newtheorem{defi}{Definition}
\newtheorem{prop}{Proposition}

\theoremstyle{definition}
\newtheorem*{rem}{Remark}

\author{\textsc{Guillaume Barraquand}\thanks{Laboratoire de Probabilités et Modèles Aléatoires, Universit\'e Paris Diderot, 5 rue Thomas Mann, 75013 PARIS. E-mail: {\tt barraquand@math.univ-paris-diderot.fr}}}
\title{\textsc{A short proof of a symmetry identity for the $(q, \mu, \nu)$-deformed Binomial distribution}}
\date{}
\begin{document}
\renewcommand{\labelitemi}{\textbullet}
\maketitle
\begin{abstract}
We give a short and elementary proof of a $(q, \mu, \nu)$-deformed Binomial distribution identity arising in the study of the $(q, \mu, \nu)$-Boson process and the $(q, \mu, \nu)$-TASEP.
This identity found by Corwin in \cite{corwin2014q} was a key technical step to prove an intertwining relation between the Markov transition matrices of these two classes of discrete-time Markov chains. This was used in turn to derive exact formulas for a large class of observables of both these processes.
\end{abstract}

\section*{Introduction}
Zero-range process and exclusion processes are generic stochastic models for transport phenomena on a lattice. Integrability of these models is an important question. 
In a short letter \cite{evans2004factorized}, Evans-Majumdar-Zia considered spatially homogeneous discrete time zero-range processes on periodic domains. They adressed and solved the question of characterizing the jump distributions for which invariant measures are product measures. 
Povolotsky \cite{povolotsky2013integrability} further examined the precise form of jump distributions allowing solvability by Bethe ansatz, and found the $(q, \mu, \nu)$-Boson process and the $(q, \mu, \nu)$-TASEP. He also conjectured exact formulas for the model on the infinite lattice.
 Using a Markov duality between the $(q, \mu, \nu)$-Boson process and the $(q, \mu, \nu)$-TASEP, Corwin \cite{corwin2014q} showed a variant of these formulas and provided a method to compute a large class of observables. 
 This can be seen as a generalization of a similar work on  $q$-TASEP and $q$-Boson process performed in \cite{borodin2012duality, borodin2013discrete}.
 In his proof, the intertwining relation between the two Markov transition matrices essentially boils down to a $(q, \mu, \nu)$-deformed Binomial distribution identity \cite[Proposition 1.2]{corwin2014q}.
  The proof was adapted from \cite[Lemma 3.7]{ borodin2013discrete} which is the $\nu=0$ case, and required the use of Heine's summation formula for the basic hypergeometric series $_2\phi_1 $. In the following, we give a short proof of this identity.

\section*{A symmetry property for the  $(q, \mu, \nu)$-deformed Binomial distribution}

First, we define the three parameter deformation of the Binomial distribution introduced in \cite{povolotsky2013integrability}.
\begin{defi}
For $\vert q\vert <1$, $0\leqslant \nu\leqslant \mu <1$ and integers $0\leqslant j\leqslant m$, define the function 
\begin{equation*}
\varphi_{q, \mu, \nu}(j\vert m ) = \mu^j \frac{(\nu/\mu;q)_j(\mu;q)_{m-j}}{(\nu ; q)_m} \binom{m}{j}_q, 
\end{equation*}
where
\begin{equation*}
\binom{m}{j}_q = \frac{\left(q ; q \right)_{m}}{\left(q ; q \right)_{j}\left(q ; q \right)_{m-j}}
\end{equation*}
are $q$-Binomial coefficients with, as usual, 
\begin{equation*}
\left(z ; q \right)_{n} = \prod_{i=0}^{n-1}\left(1-q^iz\right).
\end{equation*}
\end{defi}
It happens that for each $m\in\mathbb{N}\cup{\infty}$, this defines a probability distribution on $\lbrace 0, \dots, m \rbrace$.
\begin{lem}[Lemma 1.1, \cite{corwin2014q}]
\label{lemme}
For any $\vert q\vert <1$ and  $0\leqslant \nu\leqslant \mu <1$ , 
\begin{equation*}
\sum_{j=0}^m \varphi_{q, \mu, \nu}(j\vert m ) =1.
\end{equation*}
\end{lem}
\begin{proof}
As shown in \cite{corwin2014q},  this equation is equivalent to a specialization of  some known summation formula for basic hypergeometric series $_2\phi_1 $ (Heine's $q$-generalizations of Gauss' summation formula).
\end{proof}
This probability distribution can be seen as a $q$-analogue of the Binomial distribution, depending on two parameters $0\leqslant \nu\leqslant \mu <1$ and we call it the $(q, \mu, \nu)$-Binomial distribution. In \cite{povolotsky2013integrability}, various interesting degenerations are studied. We now state and prove the main identity. 
\begin{prop}[Proposition 1.2, \cite{corwin2014q}]
Let $X$ (resp. $Y$) be a random variable following the $(q, \mu, \nu)$-Binomial distribution on $\lbrace 0,\dots,  x\rbrace$ (resp. $\lbrace 0, \dots, y\rbrace$).
We have
\begin{equation*}
\mathbb{E}\left[ q^{xY}\right] = \mathbb{E}\left[ q^{yX}\right].
\end{equation*}
\end{prop}
\begin{proof}
Let $S_{x,y} := \sum_{j=0}^x\varphi_{q, \mu, \nu}(j\vert x ) q^{jy}$. We have to show that $S_{x,y} =S_{y,x}$ for all integers $x,y\geqslant 0$.
Our proof is based on the fact that $S_{x,y}$ satisfies a recurrence relation which is invariant when exchanging the roles of $x$ and $y$.
First notice that by lemma \ref{lemme}, $S_{x,0}=1$ for all $x\geqslant 0$, and by definition $S_{0,y}=1$ for all $y\geqslant 0$. 

The Pascal identity for q-Binomial coefficients, (see 10.0.3 in \cite{andrews2001special}), 
\begin{equation*}
\binom{x+1}{j}_q=\binom{x}{j}_q q^j+\binom{x}{j-1}_q,
\end{equation*}
 yields
\begin{eqnarray*}
S_{x+1,y} &=& \sum_{j=0}^{x+1} \mu^j \frac{(\nu/\mu;q)_j(\mu;q)_{x+1-j}}{(\nu ; q)_{x+1}} \binom{x}{j}_q q^j q^{jy} + \sum_{j=0}^{x+1} \mu^j \frac{(\nu/\mu;q)_j(\mu;q)_{x+1-j}}{(\nu ; q)_{x+1}} \binom{x}{j-1}_q q^{jy}, \\
&=& \sum_{j=0}^x \varphi_{q, \mu, \nu}(j\vert x ) \frac{1-\mu q^{x-j}}{1-\nu q^x}q^jq^{jy} + \sum_{j=0}^x \varphi_{q, \mu, \nu}(j\vert x ) \mu  \frac{1-\nu/\mu q^j}{1-\nu q^x}  q^y q^{jy}.
\end{eqnarray*}
The last equation can be rewritten 
\begin{eqnarray*}
(1-\nu q^x) S_{x+1,y} &=& \left(S_{x,y+1} -\mu q^x S_{x,y}\right ) + \left(\mu q^y (S_{x,y} - \nu/\mu S_{x,y+1}) \right),\\
&=& (1-\nu q^y)S_{x,y+1} + \mu(q^y-q^x)S_{x,y}.
\end{eqnarray*}
Thus, the sequence $\left(S_{x,y}\right)_{(x,y)\in\mathbb{N}^2}$ is completely determined by 
\begin{equation}
\left\lbrace \begin{array}{l}
(1-\nu q^x) S_{x+1,y}=(1-\nu q^y)S_{x,y+1} + \mu(q^y-q^x)S_{x,y},\\
S_{x,0}=S_{0,y}=1.
\end{array} \right.
\label{recurrence}
\end{equation}
Setting $T_{x,y} = S_{y,x}$, one notices that the sequence $\left(T_{x,y}\right)_{(x,y)\in\mathbb{N}^2}$ enjoys the same recurrence, which concludes the proof.
\end{proof}

\begin{rem}
To completely avoid the use of basic hypergeometric series, one would also need a similar proof of lemma \ref{lemme}. One can prove the result by recurrence on $m$ (as in the proof of \cite[lemma 1.3]{borodin2013discrete}), but the calculations are less elegant when $\nu\neq0$.

More precisely, fix some $m$ and suppose that for any $0\leqslant \nu\leqslant \mu <1$, $S_{m,0}(q, \mu, \nu):=\sum_{j=0}^{m} \varphi_{q, \mu, \nu}(j\vert m )=1$. Pascal's identity yields
\begin{eqnarray*}
S_{m+1,0}(q, \mu, \nu) &=&\frac{1-\mu }{1-\nu} S_{m,0}(q, q\mu, q\nu)+ \sum_{j=0}^m \varphi_{q, \mu, \nu}(j\vert m ) \mu  \frac{1-\nu/\mu q^j}{1-\nu q^m},  \\
&=&\frac{1-\mu }{1-\nu} S_{m,0}(q, q\mu, q\nu)+\frac{\mu }{1-\nu q^m}\left(S_{m,0}(q, \mu, \nu) - \nu/\mu S_{m,1}(q, \mu, \nu)\right).
\end{eqnarray*}
Then, using the recurrence formula (\ref{recurrence}) for $S_{m,1}(q, \mu, \nu)$, and applying the recurrence hypothesis, one obtains $S_{m+1,0}(q, \mu, \nu)=1$.
\end{rem}

\subsection*{Acknowledgements}
The author would like to thank his advisor Sandrine P\'ech\'e for  her support.

\bibliographystyle{amsplain}
\bibliography{qidentite.bib}

\providecommand{\bysame}{\leavevmode\hbox to3em{\hrulefill}\thinspace}
\providecommand{\MR}{\relax\ifhmode\unskip\space\fi MR }
\providecommand{\MRhref}[2]{%
  \href{http://www.ams.org/mathscinet-getitem?mr=#1}{#2}
}
\providecommand{\href}[2]{#2}
\begin{thebibliography}{1}

\bibitem{andrews2001special}
George~E Andrews, Richard Askey, and Ranjan Roy, \emph{Special functions},
  vol.~71, Cambridge University Press, 2001.

\bibitem{borodin2013discrete}
Alexei Borodin and Ivan Corwin, \emph{Discrete time q-taseps}, International
  Mathematics Research Notices (2013), rnt206.

\bibitem{borodin2012duality}
Alexei Borodin, Ivan Corwin, and Tomohiro Sasamoto, \emph{From duality to
  determinants for q-tasep and asep}, Ann. Probab., to appear, arXiv:1207.5035
  (2012).

\bibitem{corwin2014q}
Ivan Corwin, \emph{The $(q,\mu,\nu) $-boson process and $(q,\mu,\nu) $-tasep},
  arXiv preprint arXiv:1401.3321 (2014).

\bibitem{evans2004factorized}
Martin~R Evans, Satya~N Majumdar, and Royce~KP Zia, \emph{Factorized steady
  states in mass transport models}, Journal of Physics A: Mathematical and
  General \textbf{37} (2004), no.~25, L275.

\bibitem{povolotsky2013integrability}
AM~Povolotsky, \emph{On the integrability of zero-range chipping models with
  factorized steady states}, Journal of Physics A: Mathematical and Theoretical
  \textbf{46} (2013), no.~46, 465205.

\end{thebibliography}
\end{document}